\documentclass[12pt]{amsart}
\usepackage{graphicx}
\usepackage{esint}
\usepackage{amssymb, mathrsfs, url, amsfonts, amsthm, amsmath}
\usepackage{enumerate}
\usepackage{xcolor}
\usepackage{etoolbox}
\apptocmd{\sloppy}{\hbadness 10000\relax}{}{}
\apptocmd{\sloppy}{\vbadness 10000\relax}{}{}
\usepackage{hyperref}
\usepackage[letterpaper,margin=1.1in]{geometry}

\newtheorem{theorem}{Theorem}[section]
\newtheorem{lemma}[theorem]{Lemma}

\theoremstyle{definition}

\theoremstyle{remark}
\newtheorem{remark}[theorem]{Remark}

\newtheorem{question}[theorem]{Question}

\newcommand{\VMO}{\mathrm{VMO}}

\let\inf\relax \DeclareMathOperator*\inf{\vphantom{p}inf}

\newcommand{\Haus}{\mathcal{H}}
\newcommand{\RR}{\mathbb{R}}

\numberwithin{equation}{section}
\numberwithin{figure}{section}

\newcommand{\excess}{\mathop\mathrm{excess}\nolimits}
\newcommand{\HD}{\mathop\mathrm{HD}\nolimits}
\newcommand{\dist}{\mathop\mathrm{dist}\nolimits}

\newcommand{\Tan}{\mathop\mathrm{Tan}\nolimits}

\def\XXint#1#2#3{{\setbox0=\hbox{$#1{#2#3}{\int}$ }
\vcenter{\hbox{$#2#3$ }}\kern-.6\wd0}}

 \numberwithin{equation}{section}

\title[Slowly vanishing mean oscillations]{Slowly vanishing mean oscillations: non-uniqueness of blow-ups in a two-phase free boundary problem}
\date{October 30, 2022}
\author{Matthew Badger}
\author{Max Engelstein}
\author{Tatiana Toro}
\thanks{M.~Badger was partially supported by NSF DMS grant 2154047. M.~Engelstein was partially supported by NSF DMS grant 2000288 and 2143719. T.~Toro was partially supported by NSF grant DMS 1954545 and by the Craig McKibben \& Sarah Merner Professorship in Mathematics.}
\address{Department of Mathematics\\ University of Connecticut\\ Storrs, CT 06269-3009}
\email{matthew.badger@uconn.edu}
\address{Department of Mathematics\\University of Minnesota\\Minneapolis, MN, 55455 }
\email{mengelst@umn.edu}
\address{Department of Mathematics\\ University of Washington\\ Box 354350\\ Seattle, WA 98195-4350}
\email{toro@uw.edu}
\begin{document}

\begin{abstract} In Kenig and Toro's two-phase free boundary problem, one studies how the regularity of the Radon-Nikodym derivative $h=d\omega^-/d\omega^+$ of harmonic measures on complementary NTA domains controls the geometry of their common boundary. It is now known that $\log h\in C^{0,\alpha}(\partial\Omega)$ implies that pointwise the boundary has a unique blow-up, which is the zero set of a homogeneous harmonic polynomial. In this note, we give examples of domains with $\log h\in C(\partial\Omega)$ whose boundaries have points with non-unique blow-ups. Philosophically the examples arise from oscillating or rotating a blow-up limit by an infinite amount, but very slowly. \end{abstract}

\dedicatory{Dedicado a Carlos Kenig, un gran maestro y amigo en conmemoraci\'on de sus 70 a\~{n}os.}

\maketitle

\section{Introduction}

In this note, we answer a question about uniqueness of blow-ups in non-variational two-phase free boundary problems for harmonic measure \emph{in the negative}. Throughout, we let $\Omega^+=\Omega\subset\RR^n$ and $\Omega^-=\RR^n\setminus\overline{\Omega}$ denote complementary unbounded domains with a common boundary $\partial\Omega=\partial\Omega^+=\partial\Omega^-$. Furthermore, we require that $\Omega^\pm$ belong to the class of NTA domains in the sense of Jerison and Kenig \cite{jerisonandkenig}. Let $\omega^\pm$ denote harmonic measures on $\Omega^\pm$ with finite poles $X^\pm$ or with poles at infinity (see Kenig and Toro \cite{kenigtoroannals}). Finally, we assume $\omega^+\ll\omega^-\ll\omega^+$ and let \begin{equation}h=\frac{d\omega^-}{d\omega^+}\end{equation} denote the Radon-Nikodym derivative of harmonic measure on one side of the boundary with respect to harmonic measure on the other side. We are interested in understanding how different regularity assumptions on $h$ controls the geometry of $\partial\Omega$.

Following Kenig and Toro \cite{kenigtorotwophase} and Badger \cite{badgerharmonicmeasure}, we know  if $\log h\in\VMO(d\omega^+)$ (vanishing mean oscillation) or $\log h\in C(\partial\Omega)$ (continuous), then the boundary admits a finite decomposition into pairwise disjoint sets, \begin{equation}\label{decomp} \partial\Omega=\Gamma_1\cup\dots\cup\Gamma_{d_0},\end{equation} where geometric blow-ups (tangent sets) of $\partial\Omega$ centered at any $Q\in \Gamma_d$ ($1\leq d\leq d_0$) are zero sets $\Sigma_p$ of homogeneous harmonic polynomials (hhp) $p:\RR^n\rightarrow\RR$ of degree $d$. That is to say, given any boundary point $Q\in\Gamma_d$ and any sequence of scales $r_i>0$ with $\lim_{i\rightarrow\infty} r_i=0$, there exists a subsequence $r_{i_j}$ and a hhp $p$ of degree $d$ such that \begin{equation}\label{blowup-def} \lim_{j\rightarrow\infty}\max\left\{\excess\left(\frac{\partial\Omega-Q}{r_{i_j}}\cap B,\Sigma_p\right),\,\excess\left(\Sigma_p\cap B,\frac{\partial\Omega-Q}{r_{i_j}}\right)\right\}=0\end{equation} for every ball $B$ in $\RR^n$. Here $\excess(S,T)=\sup_{s\in S}\inf_{t\in T} |s-t|$ when $S,T\subset\RR^n$ are nonempty and $\excess(\emptyset,T)=0$; see \cite{localsetapproximation} for more information about this mode of convergence of closed sets (the Attouch-Wets topology). Following \cite{badgerflatpoints} and \cite{BETHarmonicpoly}, we further know that the regular set $\Gamma_1$ is relatively open, Reifenberg flat with vanishing constant, and has Hausdorff and Minkowski dimensions $n-1$, whereas the singular set $\partial\Omega\setminus \Gamma_1$ is closed and has Hausdorff and Minkowski dimension at most $n-3$.

We remark that the maximum degree $d_0$ witnessed in the decomposition \eqref{decomp} can be bounded in terms of the ambient dimension and the NTA constants of $\Omega^\pm$. When $n=2$, it is always the case that $\partial\Omega=\Gamma_1$. When $n=3$, we have $\partial\Omega=\Gamma_1\cup\Gamma_3\cup\dots\cup\Gamma_{2d_1+1}$ (odd degrees only) and for every odd $d\geq 1$, there exist two-sided domains with $\Gamma_d\neq\emptyset$. In dimensions $n\geq 4$, for every integer $d\geq 1$, even or odd, there exist two-sided domains with $\Gamma_d\neq\emptyset$. See \cite{BETHarmonicpoly} for details and \cite{AMTUR, Prats2019, Tolsa2022} for additional results on the regularity of $\Gamma_1$.

One may ask: Are the blow-ups at each point in $\partial\Omega$ unique? In other words, is the zero set $\Sigma_p$ in \eqref{blowup-def} independent of choice of the sequence of scales $r_i$? Under a stronger free boundary regularity hypothesis, the answer is \emph{affirmative}. Following Engelstein \cite{engelsteintwophase} and \cite{BETunique}, we know that if $\log h\in C^{0,\alpha}(\partial\Omega)$ for some $\alpha>0$ (H\"older continuous), then blow-ups are unique. Moreover, when $\log h\in C^{0,\alpha}(\partial\Omega)$, the regular set $\Gamma_1$ is actually a $C^{1,\alpha}$ embedded submanifold and the singular set $\partial\Omega\setminus\Gamma_1$ is $(n-3)$-rectifiable in the sense of geometric measure theory (see e.g.~\cite{Mattila}). Below, we supply examples demonstrating that under the weaker regularity hypothesis $\log h\in C(\partial\Omega)$, there may exist points in the boundary that have non-unique blow-ups.

\begin{theorem}\label{t:main} For each $d\in\{1,3\}$, there exist complementary NTA domains $\Omega^\pm\subset\RR^3$ such that $\log h\in C(\partial\Omega)$, but there exists a point in $\Gamma_d$ at which geometric blow-ups of $\partial\Omega$ are not unique.
\end{theorem}

\begin{remark}In fact, the domains that we construct below have \emph{locally finite perimeter} and \emph{Ahlfors regular} boundaries: that is, there exists $C>0$ (depending on $\Omega$) such that \begin{equation}C^{-1}r^{n-1}\leq \Haus^{n-1}(\partial\Omega\cap B(Q,r))\leq Cr^{n-1}\quad\text{for all $Q\in \partial\Omega$ and $r>0$},\end{equation} where $\Omega\subset\RR^n$ and $\Haus^{n-1}$ denotes the $(n-1)$-dimensional Hausdorff measure. Even more, the boundaries of the domains are smooth surfaces outside of a single point.\end{remark}

The basic strategy is to start with a blow-up domain $\Omega_p^\pm=\{X\in\RR^n:\pm p(X)>0\}$ associated to a hhp $p$ of degree $d$, which has  $\log h\equiv 0$ and $0\in\Gamma_d$. We then deform the domain near the origin by introducing rotations/oscillations at each scale $0<r\leq 1/100$ so that the magnitude of the oscillation at scale $r$ vanishes as $r\rightarrow 0$. The tension in the proof becomes choosing the correct speed of vanishing. On the one hand, by choosing the speed to be sufficiently \emph{quick}, we can guarantee by making estimates on elliptic measure that the deformed domain has $\log h \in C(\partial\Omega)$. On the other hand, by choosing the speed to be sufficiently \emph{slow}, we can guarantee that the deformed domain has uncountably many blow-ups at the origin, each of which are rotations of the original domain.

\begin{remark} By a suitable modification, the technique introduced in the case $d=3$ can be used to show existence of domains with $\log h\in C(\partial\Omega)$ and non-unique blow-ups at an isolated point $Q\in\Gamma_d$ for any value of $d\geq 2$. When $d\geq 3$ is odd, the examples can be produced in $\RR^3$. When $d\geq 2$ is even, the examples can be produced in $\RR^4$.\end{remark}

In a related context, Allen and Kriventsov \cite{spiral} use conformal maps to construct domains $\Omega^\pm=\{u^\pm>0\}\subset\RR^n$ ($n\geq 2$) associated to non-negative subharmonic functions $u^\pm$ for which the Alt-Caffarelli-Friedman functional \begin{equation} \Phi(r,u^+,u^-)=\frac{1}{r^4} \int_{B_r(0)}\frac{|\nabla u^+|^2}{|X|^{n-2}}\int_{B_r(0)}\frac{|\nabla u^-|^2}{|X|^{n-2}}\end{equation} has a positive limit as $r\rightarrow 0$, but whose interface $\partial\Omega=\partial\Omega^+=\partial\Omega^-$ does not have a unique tangent plane at the origin. It would be interesting to know whether a suitable modification of their examples satisfy $\log h \in C(\partial \Omega)$. For more on the connection between the ACF functional and two-phase free boundary problems for harmonic measure (originally observed by Kenig, Preiss, and Toro \cite{kenigpreisstoro}), see \cite[\S2.2]{Allen2022} and the references within.


We handle the case $d=3$ of Theorem \ref{t:main} in \S\ref{s:Domain2} and the case $d=1$ in \S\ref{s:Domain1}.

\subsection{Acknowledgments} This paper was completed while M.B. and M.E. were visiting T.T. at MSRI/SLMath in the Fall of 2022; they thank the institute for its hospitality. All three authors would also like to thank Carlos Kenig for his encouragement, kindness, and generosity over many years.

\section{The First Example: Non-Unique Singular Tangents}\label{s:Domain2}

\subsection{Description and Geometric Properties}

\begin{figure}\begin{center}\includegraphics[height=.2\textheight]{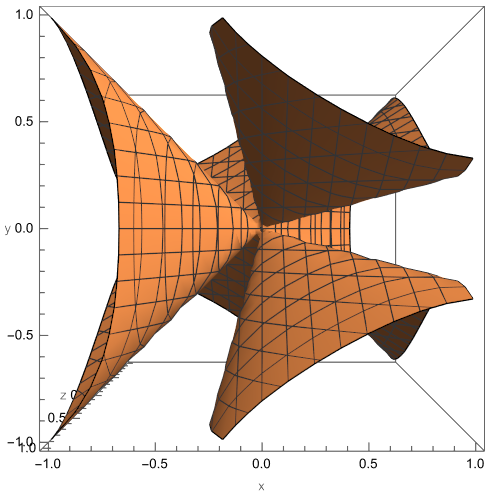}\hspace{.05\textwidth}\includegraphics[height=.22\textheight]{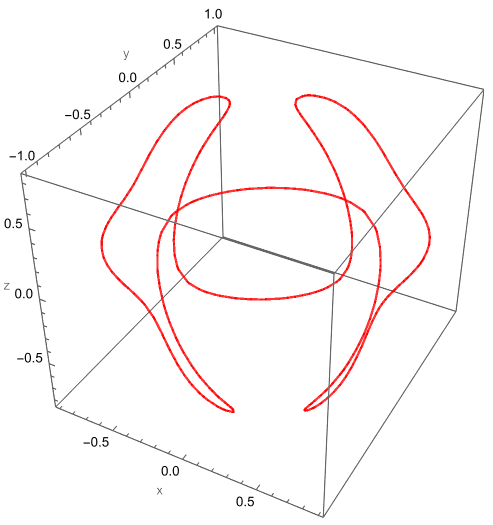}\hspace{.05\textwidth}\includegraphics[height=.2\textheight]{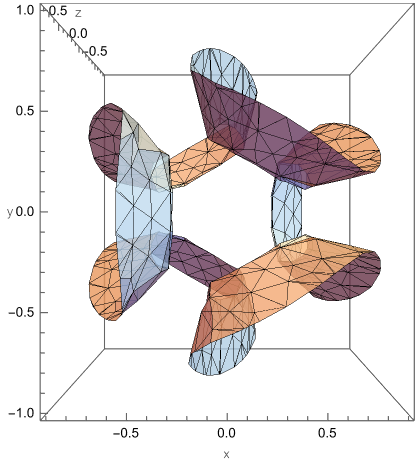}
\end{center}\caption{Left: Szulkin $\Sigma_s$, viewed from the $z$-axis. Center: the curve formed by intersection of Szulkin $\Sigma_s$ and $\mathbb{S}^2$, viewed from a different angle. Right: Szulkin $\Sigma_s$ inside of the annulus $1/2<r<1$, viewed from the $z$-axis.}\label{fig:szulkin}\end{figure}

We begin with Szulkin's example \cite{szulkin} of a degree 3 hhp, \begin{equation}\label{e:szulkindef}
s(x,y,z) = x^3 - 3xy^2 + z^3 - 1.5(x^2 + y^2)z,
\end{equation} with the interesting feature that its zero set $\Sigma_s$ is homeomorphic to $\RR^2$. See Figure \ref{fig:szulkin}.
Because $\Sigma_s$ is a cone ($s$ is homogeneous) and $\Sigma_s\cap S^2$ is a smooth curve\footnote{One can check that $\nabla s(x,y,z)=0\Leftrightarrow (x,y,z)=(0,0,0)$.}, it follows that $\Omega_s^{\pm}=\{(x,y,z)\in\RR^3:\pm s(x,y,z)>0\}$ are complementary NTA domains. Note that the positive $z$-axis belongs to $\Omega_s^+$ and the negative $z$-axis belongs to $\Omega_s^-$, since $s(0,0,\pm1)=\pm1$.

To build $\Omega^\pm$, we deform $\Omega^\pm_s$ by rotating spherical shells $\Sigma_s\cap \partial B_r(0)$ in the $xy$-plane. More precisely, we put $\Omega^{\pm} = \{\pm s_{\mathrm{twist}} > 0\}$, where $s_{\mathrm{twist}} \equiv s\circ \Phi_{-\theta}$ and $\Phi_{\pm\theta}: \mathbb R^3 \rightarrow \mathbb R^3$ are homeomorphisms given by \begin{equation}\label{e:Phidef} \Phi_{\pm\theta}(x,y,z) = (x\cos(\pm\theta) - y \sin(\pm\theta), x\sin(\pm\theta) + y\cos(\pm\theta), z),\end{equation} \begin{equation}\label{e:thetadef} \theta\equiv\theta(r) := \log(-\log(r))\quad\text{for all  }0<r:=\sqrt{x^2+y^2+z^2}\leq 1/100\end{equation} and we smoothly interpolate to $\theta(r):=0$ for all $r\geq 1$. See Figure \ref{fig:twists}.

\begin{figure}\begin{center}\includegraphics[height=.2\textheight]{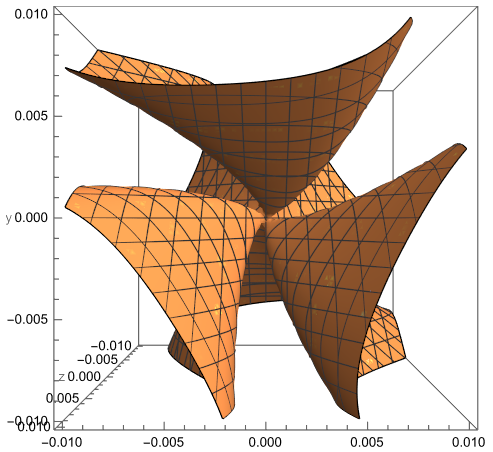}\hspace{.05\textwidth}\includegraphics[height=.2\textheight]{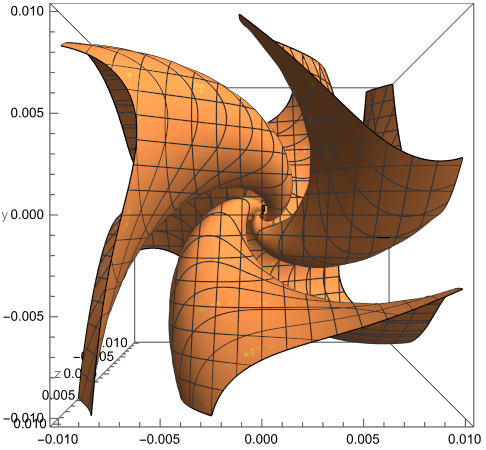}\hspace{.05\textwidth}
\includegraphics[height=.2\textheight]{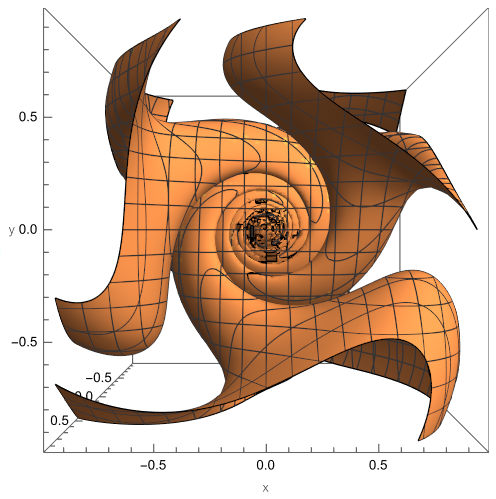}
\end{center}\caption{Examples of twisted Szulkin domains $\Omega^\pm$ defined using various rotation functions $\theta(r)$.\\
 Left: $\theta(r)=\log(-\log(r))$; the domains $\Omega^\pm$ are NTA and $\log h\in C(\partial\Omega)$.\\
  Center: $\theta(r)=-\log(r)$; the domains $\Omega^\pm$ are NTA, but $\log h\not\in \VMO(d\omega^+)$.\\
   Right: $\theta(r)=(-\log(r))^2$; the domains $\Omega^\pm$ are not NTA.}\label{fig:twists}\end{figure}

If $s_{\mathrm{twist}}(x,y,z)=0$, then $\Phi_{-\theta}(x,y,z)\in\Sigma_s$. Hence the interface $\Sigma=\partial\Omega^\pm=\Phi_{\theta}(\Sigma_s)$. Similarly, $\Omega^\pm=\Phi_{\theta}(\Omega_s^\pm)$.

\begin{remark}\label{r:thetafacts}
Let us collect some simple, but useful observations about $\theta$ and $\Phi_{\theta}$.
\begin{enumerate}
\item For any $\theta_0 \in [0, 2\pi)$, there exists a sequence $r_i \downarrow 0$ such that $\theta(r_i) = \theta_0 \pmod{2\pi}$, i.e.~such that $\min_{k\in \mathbb Z} |\theta(r_i)-\theta_0 - 2\pi k| =0$ for all $i\geq 1$.
\item For any sequence $r_i \downarrow 0$, there exists $\theta_0 \in [0, 2\pi)$ and a $r_{i_j} \downarrow 0$ such that $\theta(r_{i_j}) \rightarrow \theta_0 \pmod{2\pi}$, i.e.~$\lim_{j\rightarrow \infty}\min_{k\in\mathbb{Z}}\left|\theta(r_{i_j})-\theta_0-2\pi k\right|=0$.
\item For all $0<r\leq 1/100$, we have $|\nabla \theta| = 1/(-r\log(r))$ and $|\partial_{ij} \theta| \leq C/(-r^2\log(r))$ for all $1\leq i,j\leq 3$.
\item For all $(x,y,z)$ with $0<r\leq 1/100$, we can write $D\Phi_{\theta}=R_{\theta}+E_{\theta}$, where $$R_{\theta}=\begin{pmatrix} \cos(\theta) & -\sin(\theta) & 0 \\ \sin(\theta) & \cos(\theta) & 0 \\ 0 & 0 & 1\end{pmatrix}$$ is a rotation matrix and the ``error matrix'' $E_{\theta}$ is such that $\|E_{\theta}\|_\infty \leq C/(-\log(r))$, where the norm is the sup norm on the entries of $E_{\theta}$.
\item The map $\Phi_{\theta}: \mathbb R^3 \rightarrow \mathbb R^3$ is a quasiconformal homeomorphism, with $\Phi_{\theta}^{-1}=\Phi_{-\theta}$. Moreover, $\Phi_{\theta}$ is asymptotically conformal at the origin.
\end{enumerate}
\end{remark}

\begin{proof} The first property holds since $\theta(r)$ is continuous in $r$ and $\theta(r)\rightarrow\infty$ as $r\downarrow 0$. The second property is true by compactness of the torus $\RR/2\pi$. The third property is a straightforward computation. By another straightforward (if tedious) computation, $D\Phi_{\theta}=R_{\theta}+E_{\theta}$, where $R_{\theta}$ is as above and $E_{\theta}$ is the rank 1 matrix given by $$E_{\theta} = \begin{pmatrix} -x\sin(\theta)-y\cos(\theta) \\ x\cos(\theta)-y\sin(\theta) \\ 0\end{pmatrix}\begin{pmatrix}\theta_x & \theta_y &\theta_z\end{pmatrix}.$$ Let's examine the (1,1) entry of $E_{\theta}$. Since $\theta_x=\theta'(r)r_x=\theta'(r)x/r$ and $|x|\leq r$, we have $$|x\theta_x\sin(-\theta)+y\theta_x\cos(-\theta)|\leq 2r|\theta'(r)|\leq 2/(-\log r).$$ The other non-zero entries of $E_{\theta}$ obey the same estimate. This gives the fourth property. To prove that $\Phi_{\theta}$ is quasiconformal (see e.g.~\cite{WhatIsQC}), it suffices to check that $\Phi_{\theta}\in W^{1,n}_{\mathrm{loc}}$ and there exists $1\leq L<\infty$ such that the a.e.~defined singular values $\lambda_1\leq\lambda_2\leq \lambda_3$ of $D\Phi_{\theta}$ satisfy $\lambda_3\leq L\lambda_1$ a.e. These facts follow from property (iv) and the variational characterization of the minimum and maximum singular values. Furthermore, as $r\downarrow 0$, the maximum ratio of $\lambda_3/\lambda_1$ in $B_r$ goes to 1. Therefore, $\Phi_{\theta}$ is asymptotically conformal at the origin.
\end{proof}

The \emph{Hausdorff distance} $\HD(A,B)=\max\{\excess(A,B),\excess(B,A)\}$ for all nonempty sets $A,B\subset\RR^n$. Note that $\HD(\lambda A,\lambda B)=\lambda\HD(A,B)$ for any dilation factor $\lambda>0$.

\begin{lemma}[twisted Szulkin vs.~rotations of Szulkin] \label{l:HD} If $r,\epsilon,R>0$ and $0<Rr\leq 1/100$, then $\HD(\Sigma\cap B_{Rr},R_{\theta(r)}\Sigma_s \cap B_{Rr}) \leq C\max\big (\epsilon r, \sup\{q|\theta(q)-\theta(r)|:\epsilon r\leq q\leq Rr\}\big).$\end{lemma}

\begin{proof} For any $p\in B_{\epsilon r}$, we have $\dist(p,R_{\theta(r)}\Sigma_s\cap B_{Rr})\leq 2\epsilon r$ and $\dist(p,\Sigma\cap B_{Rr})\leq 2\epsilon r$, since $0\in R_{\theta(r)}\Sigma_s$ and $0\in \Sigma$. Thus, the main issue is to estimate distances inside $B_{Rr}\setminus B_{\epsilon r}$.

Let $p\in \Sigma\cap B_{Rr}\setminus B_{\epsilon r}$, say $p\in\Sigma \cap \partial B_{q}$ with $\epsilon r\leq q\leq Rr$. Then we may write $p=R_{\theta(q)}x$ for some $x\in\Sigma_s$. Let's estimate $\dist(p, R_{\theta(r)}\Sigma_s\cap B_{Rr})$ from above by the distance of $p$ to the point $y=R_{\theta(r)}x\in R_{\theta(r)}\Sigma_s\cap \partial B_q$. Note that $y=R_{\theta(r)}x=R_{\theta(r)}R_{-\theta(q)}p=R_{\theta(r)-\theta(q)}p$ and $|y|=|p|=q$. Hence \begin{align*}
|p-y|&\leq q|(1,0,0)-(\cos(\theta(q)-\theta(r)),\sin(\theta(q)-\theta(r)),0)|\\
&= q(2-2\cos(\theta(q)-\theta(r)))^{1/2}\\
&\leq Cq|\theta(q)-\theta(r)|,\end{align*} where the first inequality holds by geometric considerations and the last inequality used the Taylor series expansion for cosine.

A similar inequality holds starting from any $p\in R_{\theta(r)}\Sigma_s\cap B_{Rr}\setminus B_{\epsilon r}$.
\end{proof}

\begin{lemma}\label{l:secondgeometry}
With $\theta(r)=\log(-\log(r))$, the twisted Szulkin domains $\Omega^{\pm}$ as defined above are chord-arc domains (i.e.~NTA domains with Ahlfors regular boundaries). The interface $\Sigma=\partial\Omega^\pm$ has a continuum of blow-ups at the origin, each of which is a rotation of $\Sigma_s$ in the $xy$-plane.
\end{lemma}

\begin{proof} The domains $\Omega^\pm=\Phi_{\theta}(\Omega^\pm_s)$ are NTA, because global quasiconformal maps send NTA domains to NTA domains. Every boundary of an NTA domain is lower Ahlfors regular (see e.g.~\cite[Lemma 2.3]{Badger-nullsets}). Thus, $\Sigma$ is lower Ahlfors regular. To check upper Ahlfors regularity, first note that $\Sigma_s$ is upper Ahlfors regular, since $\Sigma_s$ can be covered by a finite number of Lipschitz graphs. Since $\|\det(D\Phi_{\theta})\|_\infty <\infty$, it follows that $\Sigma=\Phi_{\theta}(\Sigma_s)$ is upper Ahlfors regular, as well.

Let's address the blow-ups of $\partial \Omega$ at the origin. Let $r_i \downarrow 0$ and suppose initially that $\theta(r_i)=\theta_0 \pmod{2\pi}$ for all $i$. Let $\epsilon(r)$ be a function of $r$ to be specified below. Let $R\gg 1$ be a large radius. By Lemma \ref{l:HD}, the homogeneity of the Hausdorff distance, and the mean value theorem, we have \begin{align*}\HD(r_i^{-1}&\Sigma\cap B_{R},R_{\theta_0}\Sigma_s \cap B_{R}) \\
&\leq Cr_i^{-1}\max\big (\epsilon(r_i)r_i, \sup\{q|\theta(q)-\theta(r_i)|:\epsilon(r_i) r_i\leq q\leq Rr_i\}\big)\\
&\leq C\max\big(\epsilon(r_i), \sup\{t|\theta(tr_i)-\theta(r_i)|:\epsilon(r_i)\leq t\leq R\}\big)\\
&\leq C\max\big(\epsilon(r_i), R(R-1)r_i\sup\{|\theta'(tr_i)|:\epsilon(r_i)\leq t\leq R\}\big).\end{align*} Our task is to choose $\epsilon(r_i)$ so that \begin{equation}\label{how-to-choose-epsilon} \lim_{i\rightarrow\infty} \epsilon(r_i)=0\quad\text{and}\quad \lim_{i\rightarrow\infty}\sup\{r_i|\theta'(tr_i)|:\epsilon(r_i)\leq t\leq R\}=0.\end{equation} Since $|\theta'(r)|=1/(-r\log r)$, we have $\sup\{r_i|\theta'(tr_i)|:\epsilon(r_i)\leq t\leq R\}\leq 1/(-\epsilon(r_i)\log(Rr_i))$ for all sufficiently large $i$ (i.e.~for all sufficiently small $r_i$). Thus, \eqref{how-to-choose-epsilon} is satisfied (e.g.) by choosing $\epsilon(r) = |\log(r)|^{-1/2}$. It follows that $\lim_{i\rightarrow\infty} \HD(r_i^{-1}\Sigma\cap B_{R},R_{\theta_0}\Sigma_s \cap B_{R})=0$ for all $R>0$. This implies that $\Sigma/r_i$ converge to $R_{\theta_0}\Sigma_s$ in the sense of \eqref{blowup-def}.

In the general case, starting from any sequence $r_i\downarrow 0$, pass to a subsequence such that $\theta(r_i)\rightarrow\theta_0\pmod{2\pi}$. One can readily check that $R_{\theta(r_i)}\Sigma_s$ converges to $R_{\theta_0}\Sigma_s$ in the Attouch-Wets topology. Therefore, $\Sigma/r_i$ converges to $R_{\theta_0}\Sigma_s$ in the sense of \eqref{blowup-def} by the special case and the triangle inequality for excess.
\end{proof}

\begin{remark}For all exponents $0<p<1$, the twisted Szulkin domains defined using the rotation function $\theta(r)=(-\log(r))^p$ also satisfy the conclusions of Lemma \ref{l:secondgeometry}. However, there is phase transition at $p=1$. When $\theta(r)=-\log(r)$, one can show that the blow-ups of $\Sigma$ are no longer zero sets of hhp. The essential difference is that the ``speed of rotation'' vanishes as one zooms-in at the origin when $p<1$, but the ``speed of rotation'' is constant when $p=1$. When $p>1$, the ``speed of rotation'' goes to infinity as one zooms-in at the origin and the associated twisted Szulkin domains $\Omega^\pm$ are not even NTA. See Figure \ref{fig:twists}.\end{remark}

\subsection{Potential Theory for the First Example}\label{ss:potentialsecond}

Let $r_i \downarrow 0$ be an arbitrary sequence of radii going to zero and let $K \gg 1$. Recall that $\Sigma \cap (B_{Kr_i}\backslash B_{r_i/K}) = \Phi_{\theta}(\Sigma_s \cap (B_{Kr_i} \backslash B_{r_i/K})).$ Set \begin{equation}\label{rescaled-u}\tilde{u}^\pm_i(x) =  \frac{u^{\pm}\circ \Phi_{-\theta}^{-1}(r_ix)r_i}{\omega^{\pm}(B_{r_i})},\end{equation} where $u^{\pm}$ are the Green's functions with poles at infinity for $\Omega^{\pm}$. Then in $\Omega_s^\pm \cap B_K\backslash B_{1/K}$, we have that  $\tilde{u}^{\pm}_i$ satisfies $$-\mathrm{div}(B(r_ix)\nabla -) = 0,\qquad B = (\mathrm{det}D\Phi_\theta)^{-1}(D\Phi_\theta)(D\Phi_\theta)^T$$ and $\Phi_\theta$ is as in \eqref{e:Phidef}.

To see that $B(r_ix)$ is Lipschitz regular, we note that Remark \ref{r:thetafacts}(iii) implies that $\|DB\| \leq \frac{C}{r\log(r)}$. Therefore, using the fundamental theorem of calculus along curves which stay in the annulus $B_K \backslash B_{1/K}$ \begin{equation}\label{eq:bbound} \|B(r_ix) - B(r_iy)\| \leq Cr_i|x-y|\sup_{B_{Kr_i}\backslash B_{r_i/K}} \|DB\| \leq \frac{CK}{|\log(r_i)|}|x-y|, \forall x,y \in B_K\backslash B_{1/K},\end{equation} where $C > 0$ is independent of $i,K$.  This uniform Lipschitz continuity immediately implies the next result:



\begin{lemma}\label{l:c1alphaconv}
Let $\alpha \in (0,1), K > 1$. The sequence $\tilde{u}^{\pm}_i$ is pre-compact in $C^{1,\alpha}(\Omega^\pm_s\cap B_{K}\setminus B_{1/K})$. Furthermore, there exists a subsequence along which $\tilde{u}^{\pm}_i \rightarrow \kappa s$, uniformly on compacta, where $s$ is the Szulkin polynomial, for some $\kappa>0$.
\end{lemma}

\begin{proof}
We see that $\tilde{u}^\pm_i$ solves an elliptic PDE with coefficients that are Lipschitz continuous and elliptic with coefficients independent of $i$. Furthermore, $$\sup_{B_{4K}} |\tilde{u}^\pm_i| \leq C\Leftrightarrow \sup_{B_{4Kr_i}} |u^+| \leq C\frac{\omega^+(B_{r_i})}{r_i}.$$ The latter inequality holds (with a $C > 0$ that depends on $K$) by the Caffarelli-Fabes-Mortola-Salsa and doubling estimates on harmonic measure in NTA domains, see e.g. \cite{jerisonandkenig}. Then Schauder theory tells us that $\tilde{u}^\pm_i$ are uniformly in $C^{1,\alpha}(\overline{\Omega_s^+} \cap B_K \backslash B_{1/K})$ for any $\alpha \in (0,1)$; see \cite[Theorem 8.3]{GT-book}. The precompactness follows.

Passing to a subsequence, we get that the sequences converges to functions $\tilde{u}^\pm_\infty$, which solves $-\mathrm{div}(B_\infty \nabla \tilde{u}^\pm_\infty) = 0$ in $\Omega_s^\pm\cap B_{K}\backslash B_{1/K}$. From \eqref{eq:bbound} we see that $B_{\infty} = \mathrm{Id}$ and so, invoking a diagonal argument,  $\tilde{u}^\pm_i \rightarrow \tilde{u}^\pm_\infty$, uniformly on compacta in $\mathbb R^3$. Furthermore, $\tilde{u}^{\pm}_\infty$ are positive harmonic functions in $\Omega^{\pm}_s$ that vanish on $(\Omega^{\pm}_s)^c$.

Since $(\Omega^{\pm}_s)^c$ are (global) NTA domains, the boundary Harnack inequality implies that there are scalars $\kappa_{\pm} > 0$ such that $\tilde{u}^{\pm}_\infty = \kappa_{\pm} s$ (see \cite[Lemma 3.7 and Corollary 3.2]{kenigtoroannals}). 

To wrap up, let us again note that the points $(0,0,\pm 1) \in \Omega_s^\pm$ are invariant under $\Phi_\theta$. Furthermore by symmetry $u^+(0,0,1) = u^-(0,0,-1)$ and $\omega^+(B_r) = \omega^-(B_r)$ for all $r$. Thus, $u_\infty^+(0,0,1) = u_\infty^-(0,0,-1)$ and this number determines the constant of proportionality with $s$.
\end{proof}

Finally, the proof of the continuity of $\log h$ follows immediately:

\begin{proof}[Proof of $\log h \in C(\partial \Omega)$]
We note that away from the origin, $\partial \Omega$ is smooth so continuity of the Radon-Nikodym derivative follows from classical potential theory. Furthermore, arguing by symmetry (that is, $-\Omega^+ = \Omega^-$) we have that $\omega^{+}(B(0,r)) = \omega^-(B(0,r))$ for all $r > 0$. Thus, recalling that $u^{\pm}$ are the Green's function for $\Omega^{\pm}$ respectively, we are done if we can show that  $$\lim_{\partial \Omega \ni Q \rightarrow 0} \frac{|\nabla u^+|(Q)}{|\nabla u^-|(Q)} = 1.$$ (Recall that where $\partial \Omega$ is smooth, $C^{1,\alpha}$ is sufficient, the Radon-Nikodym derivative is given by the ratio of the derivatives of the Green functions \cite{Kellogg12}).

Let $Q_i \in \partial \Omega$ with $Q_i \rightarrow 0$ and let $|Q_i| = r_i \downarrow 0$. Let $\tilde{u}^\pm_i$ be given by \eqref{rescaled-u}. Then $$\frac{\omega^{\pm}(B_{r_i})}{r_i^2} D\Phi_\theta(r_ix)\nabla \tilde{u}^{\pm}_i(x) = \nabla u^{\pm}(\Phi_{-\theta}^{-1}(r_ix)).$$ Let $\tilde{Q}_i = \Phi_\theta(Q_i)/r_i \in \Sigma_s \cap \partial B_1$. We have shown that $$\frac{|\nabla u^+|(Q_i)}{|\nabla u^-|(Q_i)}
= \frac{|D\Phi_\theta(r_i \tilde{Q}_i) \nabla \tilde{u}^{+}_i(\tilde{Q}_i)|}{|D\Phi_\theta(r_i\tilde{Q}_i)\nabla \tilde{u}^{-}_i(\tilde{Q}_i)|}.$$ Continuity of $\log h$ follows from Lemma \ref{l:c1alphaconv} (the lemma implies that $\tilde{u}^{\pm} \rightarrow \kappa s$ in $C^{1,\alpha}(\overline{\Omega_s}\cap B_2\backslash B_{1/2})$) and the fact that along some subsequence $D\Phi_\theta(r_ix) \rightarrow R_{\theta_0}$ for some $\theta_0$ (depending on the subsequence).
\end{proof}

\section{The Second Example: Non-Unique Flat Tangents}\label{s:Domain1}

\subsection{Description and Geometric Properties}

To show non-uniqueness at ``flat points" we adapt an example from \cite{ToroJDG}. We set $\Omega^{\pm}=\{(x,y,z)\in\RR^3: \pm (z -v(x,y)) > 0\},$ where $v:\RR^2\rightarrow \RR$ is defined by setting $v(0,0)=0$,  $$v(x,y) = x \log|\log(r)| \sin(\log|\log(r)|)\quad\text{when } 0<r=(x^2+y^2)^{1/2} \leq 1/100,$$ and smoothly (e.g. $C^{1,\alpha}$) interpolating to $v(x,y)=1$ when $r \geq 1$.

\begin{lemma}[see {\cite[Example 2]{ToroJDG}}]\label{l:propertiesofsurface}
The graph domains $\Omega^\pm$ are chord-arc domains. The interface $\Sigma=\partial \Omega^\pm$ has a continuum of blow-ups at the origin, each of which is a plane $z=mx$ with ``slope'' $-\infty\leq m\leq \infty$.
\end{lemma}

%

\begin{figure}\begin{center}\includegraphics[height=.15\textheight]{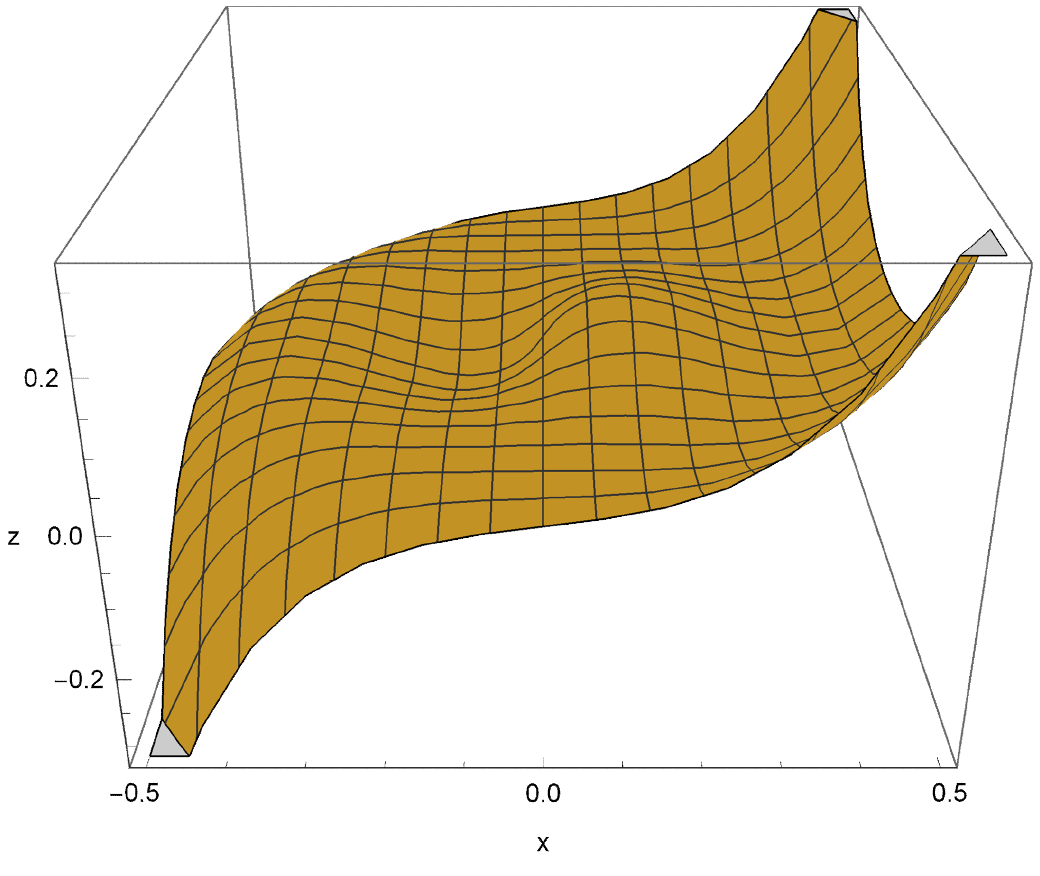}\hspace{.05\textwidth}\includegraphics[height=.15\textheight]{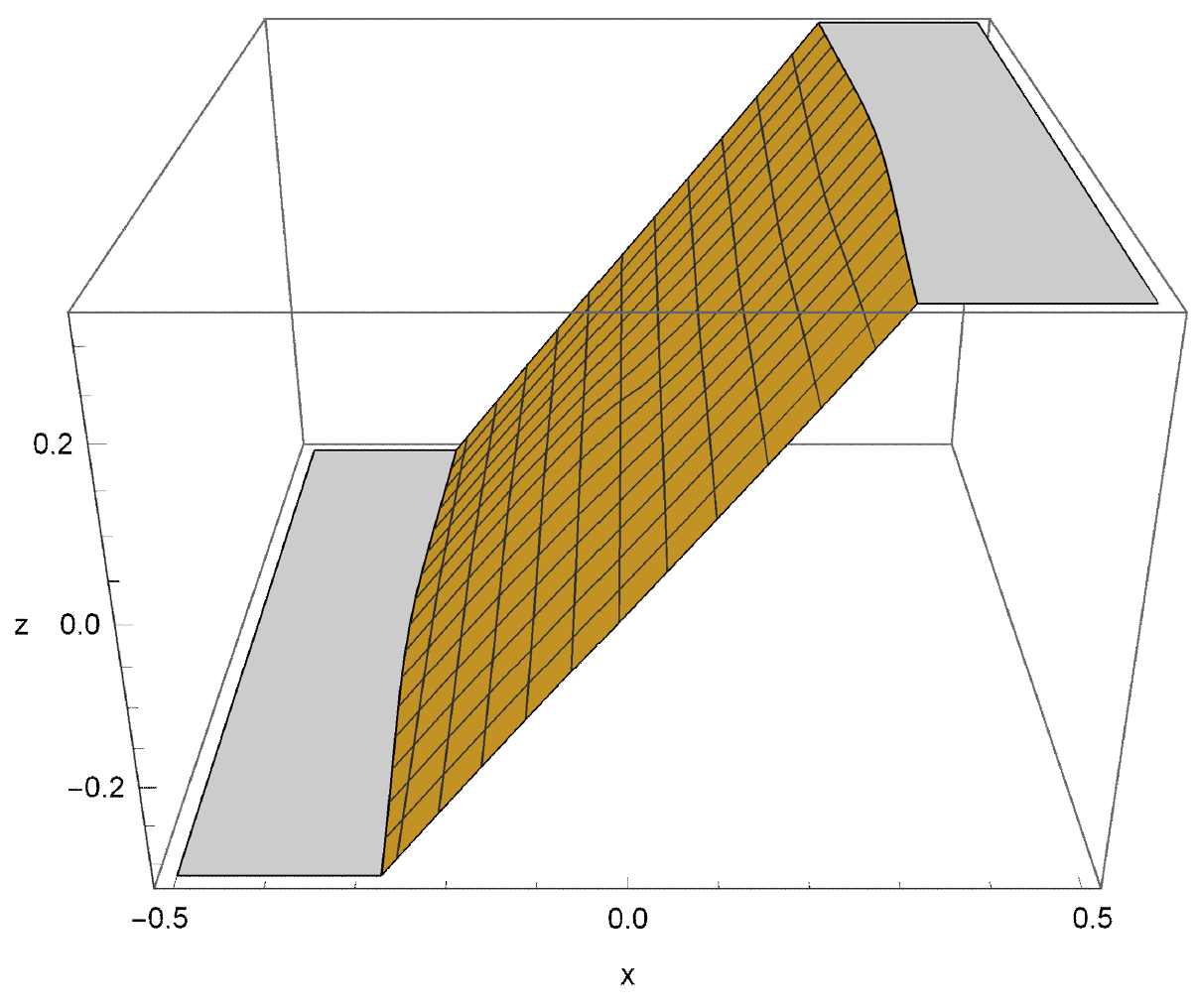}\hspace{.05\textwidth}
\includegraphics[height=.15\textheight]{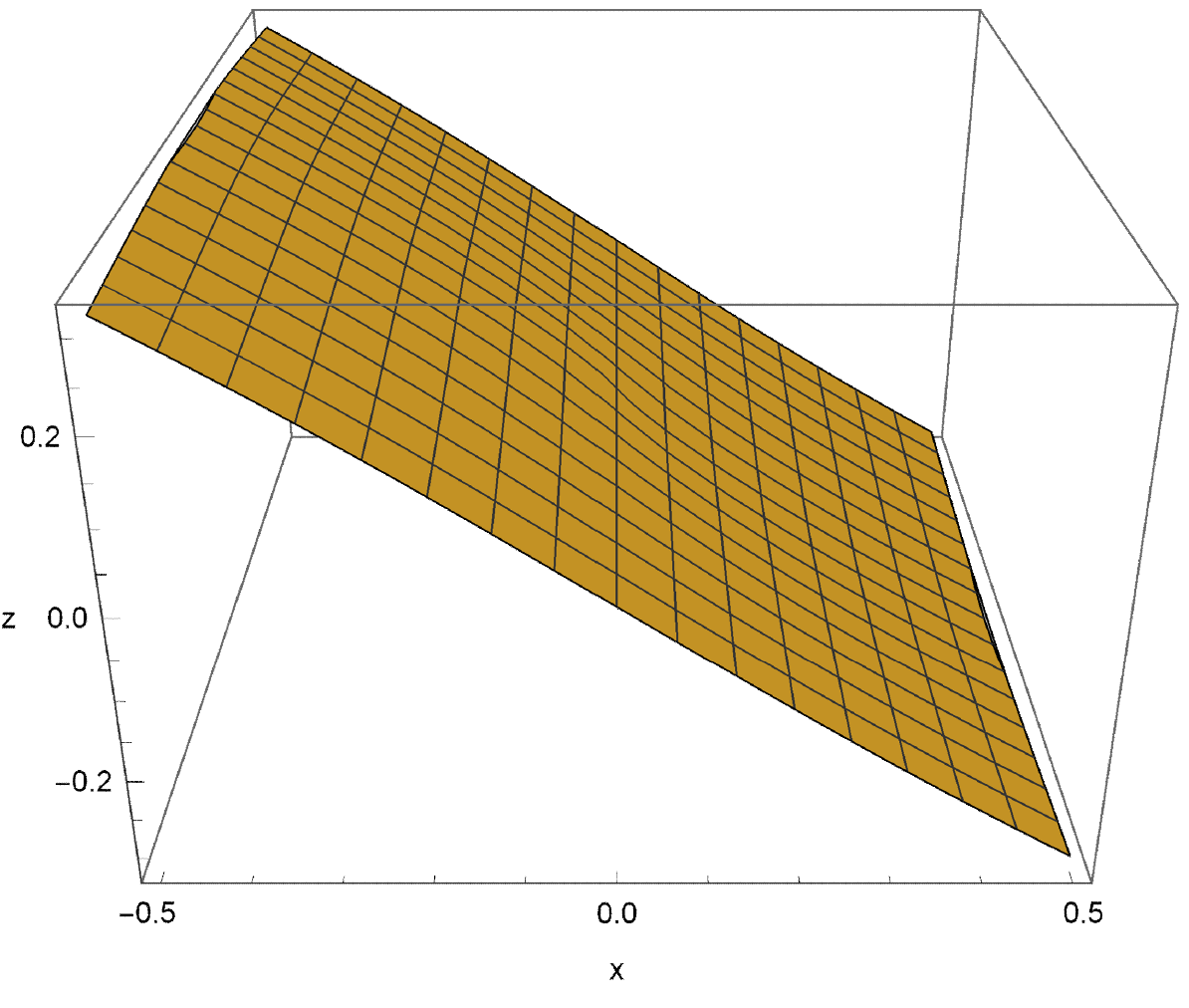}
\end{center}\caption{Blow-ups $\Sigma/r$ of the interface $\Sigma=\partial\Omega^\pm$ of the graph domains. Left: $r=1$. Center: $r=10^{-6}$. Right: $r=10^{-12}$.}\label{fig:graphs}\end{figure}

\begin{remark}\label{rem:actuallyvanishing}
Moreover, $\Omega^{\pm}$ are vanishing chord-arc domains in the sense of \cite{kenigtoroannsci}. This can be seen as follows. First, every pseudo blow-up (an Attouch-Wets limit $\Gamma$ of $(\Sigma-Q_i)/r_i$ with $Q_i\rightarrow Q$ and $r_i \downarrow 0$) is a plane. Indeed, on the one hand, if $\limsup_{i\rightarrow\infty}|Q_i-Q|/r_i =\infty$, then $\Gamma$ is a plane, because $\Sigma\setminus\{0\}$ is smooth. On the other hand, if $|Q_i|/r_i \leq C$ for all $i$, then $\Gamma$ is a translate of a blow-up at $Q$ (see \cite[Lemma 3.7]{localsetapproximation}), and thus, $\Gamma$ is a plane by Lemma \ref{l:propertiesofsurface}. Because every pseudo blow-up is a plane, $\Sigma$ is locally Reifenberg vanishing. Now, $v\in W^{2,2}(\RR^2)$ (see \cite{ToroJDG}). Hence, by Sobolev embedding, the normal vector of the interface $\hat{n} \in \mathrm{BMO}(\partial \Omega)$ with small BMO norm. Therefore, $\Omega^\pm$ are vanishing chord-arc domains; see e.g.~\cite{kenigtoroduke, BEGTZ}.
\end{remark}

\subsection{Potential Theory for the Second Example}

Following the approach of \S\ref{ss:potentialsecond}, we now prove that $\log h \in C(\partial \Omega)$.\footnote{One could prove the weaker result that $\log h \in \mathrm{VMO}(d\omega^+)$ using Remark \ref{rem:actuallyvanishing} and standard properties of $A_\infty$ weights.} As before, because $\partial\Omega$ is smooth outside of any neighborhood of the origin, $\log h \in C^\infty$ on $\partial\Omega\setminus B_r(0)$ for any $r>0$. Thus, the key point is to show that $\log h$ is continuous at the origin.

Let $H^\pm=\{\pm z>0\}$ denote the open upper and lower half-spaces. Let $r_i \downarrow 0$ be arbitrary, $K \gg 1$ and write $$\{z= v(x,y)\}\cap (B_{Kr_i}\backslash B_{r_i/K}) = \Phi^{-1}(\{z = 0\} \cap (B_{Kr_i} \backslash B_{r_i/K})),$$ where $\Phi:\RR^3\rightarrow\RR^3$ is the homeomorphism given by \begin{equation}\label{e:Dphi2} \Phi(x,y,z) \equiv (x,y, z-v(x,y)).\end{equation} Set $\tilde{u}^\pm_i(p) =  \frac{u^{\pm}\circ \Phi^{-1}(r_ip)r_i}{\omega^{\pm}(B_{r_i}(0))}$, where $u^{\pm}$ are the Green's functions with poles at infinity for $\Omega^{\pm}$, and the $\omega^{\pm}$ are the corresponding harmonic measures. In $H^\pm\cap B_K\backslash B_{1/K}$, $\tilde{u}^{\pm}_i$ satisfies $$-\mathrm{div}(B(r_ip)\nabla \tilde{u}^{\pm}_i(p)) = 0,\qquad B = (\mathrm{det}D\Phi)^{-1}(D\Phi)(D\Phi)^T.$$





\begin{lemma}\label{l:c1alphaconv2}
Let $\alpha \in (0,1), K > 1$. The sequence $\tilde{u}^{\pm}_i$ is pre-compact in $C^{1,\alpha}(\overline{H^\pm}\cap B_K\backslash B_{1/K})$. Furthermore, there exists a subsequence along which $\tilde{u}^{\pm}_i \rightarrow \kappa z_{\pm}$ for some $\kappa>0$ uniformly on compact subsets of $\mathbb R^3$.
\end{lemma}

\begin{proof} We claim that $\tilde{u}^\pm_i$ solves an elliptic PDE with Lipschitz continuous coefficients in $B_K\backslash B_{1/K}\cap H^{\pm}$.  Indeed, \begin{equation}\label{e:estB2}|B(r_ip) - B(r_i q)| \leq Cr_i |p-q| \|DB\|_{L^\infty(B_{Kr_i}\backslash B_{r_i/K})} \stackrel{\text{\cite{ToroJDG}}}{\leq} CKr_i \frac{\log|\log(r_i)|}{r_i|\log(r_i)|}|p-q| \leq CK|p-q|,\end{equation} by the fundamental theorem of calculus.

Arguing as in Lemma \ref{l:c1alphaconv} above, $\tilde{u}^\pm_i$ are uniformly in $C^{1,\alpha}(\overline{H^+}\cap B_K\backslash B_{1/K})$ for any $\alpha \in (0,1)$ and thus have the desired pre-compactness. Passing to a subsequence and invoking a diagonal argument $\tilde{u}^\pm_i \rightarrow \tilde{u}^\pm_\infty$ uniformly on compacta. Furthermore $\tilde{u}^{\pm}_\infty > 0$ and solves $-\mathrm{div}(B_\infty \nabla \tilde{u}^\pm_\infty) = 0$ in $H^\pm$ and has $\tilde{u}^\pm_\infty(x,y,0) = 0$. We see in \eqref{e:estB2} that $B_\infty$ is constant (as $\log|\log(r_i)|/\log(r_i) \downarrow 0$) and so $-\mathrm{div}(B_\infty \nabla z) = 0$. Again, up to scalar multiplication there is a unique signed solution of $-\mathrm{div}(B_\infty \nabla -) = 0$ in $H^{\pm}$ which vanishes on $\{z =0\}$ and that has subexponential growth at infinity. Continuing to follow the argument for Lemma \ref{l:c1alphaconv}, we conclude that $\tilde{u}_\infty^\pm = \kappa_\pm z_\pm$, with $\kappa_+ = \kappa_-$. (Remember that $-\{z > v(x,y)\} = \{z < v(x,y)\}$, because $v$ is odd.)
\end{proof}

Finally, the proof of the continuity of $\log h$ in this context follows exactly as in \S\ref{ss:potentialsecond} except that we must be more careful estimating $|D\Phi(r_i\tilde{Q}_i)\nabla \tilde{u}^\pm(\tilde{Q}_i)|$. (We do not know that $D\Phi(r_ip)$ converges to a rotation as $r_i \downarrow 0$.) However, observe that $\tilde{u}^{\pm} \equiv 0$ on $\{z =0\}$, so we know that $\nabla \tilde{u}^\pm(\tilde{Q}_i)$ is parallel to $e_3$. Thus, an elementary computation shows that $$\frac{|D\Phi(r_i\tilde{Q}_i)\nabla \tilde{u}^+(\tilde{Q}_i)|}{|D\Phi(r_i\tilde{Q}_i)\nabla \tilde{u}^-(\tilde{Q}_i)|} = \frac{|\nabla \tilde{u}^+(\tilde{Q}_i)| |D\Phi(r_i\tilde{Q}_i)e_3|}{|\nabla \tilde{u}^-(\tilde{Q}_i)||D\Phi(r_i\tilde{Q}_i)e_3|} = \frac{|\nabla \tilde{u}^+(\tilde{Q}_i)|}{|\nabla \tilde{u}^-(\tilde{Q}_i)|}.$$ The quantity on the right hand side converges to 1 by Lemma \ref{l:c1alphaconv2}. As in \S\ref{ss:potentialsecond}, it follows that $\log h\in C(\partial\Omega)$.

\section{Open Questions and Further Directions}\label{s:moreQs}

We end by presenting some natural open questions. Our first question concerns the size of the set of non-uniqueness:

\begin{question}\label{q:size}
Let $\Omega^{\pm}\subset\RR^n$ be complementary NTA domains with $\log h \in C(\partial \Omega)$. Is it possible for $$NU(\Omega) :=\{Q\in \partial \Omega : \text{there is no unique (geometric) blow-up at } Q\}$$ to have Hausdorff dimension $n-1$?
\end{question}

We note that a local version of \cite[Theorem 1.1]{Tolsa2022} implies that the set $\Gamma_1$ of flat points in $\partial\Omega$ is uniformly rectifiable. Thus $\omega^{\pm}(NU) = 0 = \mathcal H^{n-1}(NU\cap \Gamma_1)$. Further, by the main result of \cite{BETHarmonicpoly}, $\dim \partial \Omega \backslash \Gamma_1 \leq n-3$. Thus, $\mathcal H^{n-1}(NU) = 0$. On the other hand, the example of \cite{spiral} suggests that $\mathcal H^{n-2}(NU\cap \Gamma_1) > 0$ may be possible.

The example in \S\ref{s:Domain2} (twisted Szulkin) shows that it is possible for all singular points to have non-unique blowups and for the set of singular points with non-unique blowups to have positive $\mathcal H^{n-3}$-measure. (When $n\geq 4$, simply take $\Omega^\pm\times \RR^{n-3}$.) This is sharp by \cite{BETHarmonicpoly}. Thus, the natural analogue of Question \ref{q:size} is answered in the affirmative.

Our second question asks what are the possible tangent cones at points of non-unique blow-up:

\begin{question}\label{q:moduli1}
Let $C \subset G(n, n-1)$ be a compact, connected subset of the Grassmannian. Does there exist a pair of complementary NTA domains $\Omega^{\pm}$ with $\log h \in C(\partial \Omega)$ and a point $Q\in\partial\Omega$ at which $\Tan(\partial\Omega,Q)=C$?
\end{question}

In \S\ref{s:Domain1}, we showed that the set $\Tan(\partial\Omega,0)$ of blow-ups of the interface of the graph domains at the origin consists of all planes $z= mx$ with ``slope'' $-\infty\leq m \leq +\infty$. For any closed interval $I \subset \mathbb R$, it is not hard to adapt the example so that the blowups at the origin are exactly the planes $z = mx$ with $m \in I$. It is known that for any closed set $\Sigma\subset\RR^n$ and $Q\in\Sigma$, the set $\Tan(\Sigma,Q)$ of all tangent sets of $\Sigma$ at $Q$ is closed and connected in the Attouch-Wets topology \cite{localsetapproximation}; the statement and proof of this fact was originally motivated by similar statement for tangent measures \cite{preiss, kenigpreisstoro}.

We may also ask a version of Question \ref{q:moduli1} at points where the blow-ups are homogeneous of higher degree:

\begin{question}\label{q:moduli2}
Let $\mathscr H_{n,d}$ be the set of degree $d$ homogeneous harmonic polynomials $p$ in $\RR^n$ such that $\Omega^\pm_p=\{\pm p>0\}$ are NTA domains. For each $n\geq 3$ and $d\geq 2$ and $C \subset \mathscr H_{n,d}$, which is compact and connected, does there exist complementary NTA domains $\Omega^{\pm}$ with $\log h\in C(\partial \Omega)$ and a point $Q\in \partial \Omega$ at which $\Tan(\partial\Omega,Q)=\{\Sigma_p:p\in C\}$?
\end{question}

 The condition that $\mathbb R^n\backslash \Sigma_p$ is a union of two NTA domains is necessary for $\Sigma_p$ to arise as a blow-up of the interface of complementary NTA domains. The first step to answering Question \ref{q:moduli2} may be to study the ``moduli space'' of $\mathscr H_{n,d}$ when $d \geq 2$. For example:

\begin{question}\label{q:moduli3}
If $p$ and $q$ lie in the same connected component of $\mathscr{H}_{n,d}$, is it true that $\Sigma_q$ is bi-Lipschitz equivalent to $\Sigma_p$?
\end{question}

\bibliography{nonunique}{}
\bibliographystyle{amsbeta}
\end{document}